\documentclass[a4paper]{article}

\usepackage{amsmath}
\usepackage{amssymb}
\usepackage{amsthm}
\usepackage{graphicx}
\usepackage[OT4]{fontenc}
\usepackage[utf8]{inputenc}
\usepackage{tikz}
\usepackage{url}
\usepackage[numbers]{natbib}
\usepackage{bm}
\allowdisplaybreaks

\newcommand{\abs}[1]{\left| #1 \right|}
\newcommand{\okra}[1]{\left( #1 \right)}
\newcommand{\kwad}[1]{\left[ #1 \right]}
\newcommand{\klam}[1]{\left\{ #1 \right\}}
\DeclareMathOperator{\mean}{\mathbf{E}}
\DeclareMathOperator{\PPM}{PPM}
\DeclareMathOperator*{\argmax}{arg\, max}
\newcommand{\boole}[1]{{\bf 1}{\klam{#1}}}

\newtheorem{Definition}{Definition}
\newtheorem{Proposition}{Proposition}
\newtheorem{Theorem}{Theorem}

\author{{\L}ukasz D\k{e}bowski and Tomasz Steifer\thanks{%
    T. Steifer is with the Institute of Fundamental Technological
    Research, Polish Academy of Sciences, ul. Pawi\'nskiego 5B, 02-106
    Warszawa, Poland (e-mail: tsteifer@ippt.pan.pl). %
    {\L}. D\k{e}bowski is with the Institute of Computer Science,
    Polish Academy of Sciences, ul. Jana Kazimierza 5, 01-248
    Warszawa, Poland (e-mail: ldebowsk@ipipan.waw.pl). % 
  }
}

\title{Universal Coding and Prediction \\ on Ergodic Martin-L\"of
  Random Points} \date{}

\begin{document}

\maketitle

\begin{abstract}
  Suppose that we have a method which estimates the conditional
  probabilities of some unknown stochastic source and we use it to
  guess which of the outcomes will happen. We want to make a correct
  guess as often as it is possible. What estimators are good for this?
  In this work, we consider estimators given by a familiar notion of
  universal coding for stationary ergodic measures, while working in
  the framework of algorithmic randomness, i.e, we are particularly
  interested in prediction of Martin-L\"of random points. We outline
  the general theory and exhibit some counterexamples.  Completing a
  result of Ryabko from 2009 we also show that universal probability
  measure in the sense of universal coding induces a universal
  predictor in the prequential sense. Surprisingly, this implication
  holds true provided the universal measure does not ascribe too low
  conditional probabilities to individual symbols. As an example, we
  show that the Prediction by Partial Matching (PPM) measure satisfies
  this requirement with a large reserve.
    \\
    \textbf{Keywords}:
    algorithmic randomness; stationary ergodic processes; universal
    coding; universal prediction; prediction by partial matching
    \\
    \textbf{MSC 2000}:
    94A29, 62M20, 03D32
 \end{abstract}

% %MSC 2000: 
% % 94A29 - Source coding
% % 60G10 - Stationary processes
% % 94A17 - Measures of information, entropy
% % 03D32 - Algorithmic randomness and dimension
% % 68Q30 - Algorithmic information theory (Kolmogorov complexity, etc.)
% % 62M20 - Inference from stochastic processes and prediction

% \end{titlepage}
% \pagestyle{plain}   

% \tableofcontents

\section{Introduction}
\label{secIntro}

A sequence of outcomes $X_1,X_2,\ldots$ coming from a finite alphabet
is drawn in a sequential manner from an unknown stochastic source
$P$. At each moment a finite prefix $X_1^n=(X_1,X_2,\ldots,X_n)$ is
available.  The forecaster has to predict the next outcome using this
information. The task may take one of the two following forms. In the
first scenario, the forecaster simply makes a guess about the next
outcome. The forecaster's performance is then assessed by comparing
the guess with the outcome. This scenario satisfies the weak
prequential principle of Dawid \cite{dawid1999prequential}. In the
second case, we allow the forecaster to be uncertain, namely, we ask
them to assign a probability value for each of the outcomes. These
values may be interpreted as estimates of the conditional
probabilities $P(X_{n+1}|X_1^n)$. Various criteria of success may be
chosen here such as the quadratic difference of distributions or the
Kullback-Leibler divergence. The key aspect of both problems is that
we assume limited knowledge about the true probabilities governing the
process that we want to forecast. Thus, an admissible solution should
achieve the optimal results for an arbitrary process from some general
class. For clarity, term ``universal predictor'' will be used to
denote the solution of guessing the outcome, while the solution of
estimation of probabilities will be referred to as ``universal
estimator'', ``universal measure'', or ``universal code'' depending on
the exact meaning.

In the literature, it is frequently assumed without explanation that
universal predictors and universal estimators reduce to another in
some way.  Such simplification is not without a rationale. It is known
that for a fixed stochastic source $P$, the optimal prediction is
given by the predictor induced by $P$, i.e., the informed scheme which
predicts the outcome with the largest conditional probability
$P(X_{n+1}|X_1^n)$ \cite{Algoet94}. In particular, a good universal
estimator should induce a good universal predictor. That being said,
the devil is hidden in the details such as what is meant by a ``good''
universal code, measure, estimator, or predictor.

In this paper, we will assume that the unknown stochastic source $P$
lies in the class of stationary ergodic measures. Moreover, we are
concerned with universal measures which are universal in the
information-theoretic sense of universal coding, i.e., the rate of
Kullback-Leibler divergence of the estimate and the true measure $P$
vanishes for any stationary ergodic measure. As for universal
predictors, we assume that the rate of correct guesses is equal to the
respective rate for the predictor induced by measure $P$.  In this
setting, a universal measure need not belong to the class of
stationary ergodic measures and can be computable, which makes the
problem utterly practical. Our framework should be contrasted with
universal prediction \`a la Solomonoff for left-c.e.\ semimeasures
where the universal semimeasure belongs to the class and is not
computable \cite{Solomonoff64}. In general, existence of a universal
measure for an arbitrary class of probability measures can be linked
to separability of the considered class \cite{Ryabko10b}.

Now, we can ask the question whether a universal measure in the above
sense of universal coding induces a universal predictor. Curiously,
this simple question has not been unambiguously answered in the
literature, although a host of related propositions were compiled by
Suzuki \cite{Suzuki03} and Ryabko \cite{Ryabko08,Ryabko09}, see also
\cite{RyabkoAstolaMalyutov16}.  It was shown by Ryabko
\cite{Ryabko09}, see also \cite{RyabkoAstolaMalyutov16}, that the
expected value of the average absolute difference between the
conditional probability for a universal measure and the true value
$P(X_{n+1}|X_1^n)$ converges to zero almost surely for any stationary
ergodic measure $P$.  Ryabko \cite{Ryabko09} showed also that there
exists a universal measure that induces a universal predictor. As we
argue in this paper, this result does not solve the general problem.

Completing works \cite{Suzuki03,Ryabko08,Ryabko09}, in this paper we
will show that any universal measure $R$ in the sense of universal
coding that additionally satisfies a uniform bound
\begin{align}
  \label{CondDominationIntro}
  -\log R(X_{n+1}|X_1^n)&\le\epsilon_n\sqrt{n/\ln n},
                          \quad \lim_{n\to\infty} \epsilon_n=0
\end{align}
induces a universal predictor, indeed. On our way, we will use the
Breiman ergodic theorem \cite{Breiman57} and the Azuma inequality for
martingales with bounded increments \cite{Azuma67}, which is the
source of condition (\ref{CondDominationIntro}).  It is left open
whether this condition is necessary. Fortunately, condition
(\ref{CondDominationIntro}) is satisfied by reasonable universal
measures such as the Prediction by Partial Matching (PPM) measure
\cite{ClearyWitten84,Ryabko88en2,Ryabko08}, which we also show in this
paper. It may be interesting to exhibit universal measures for which
this condition fails. There is a large gap between bound
(\ref{CondDominationIntro}) and the respective bound for the PPM
measure, which begs for further research.

To add more weight and to make the problem interesting from a
computational perspective, we consider this topic in the context of
algorithmic randomness and we seek for effective versions of
probabilistic statements. Effectivization is meant as the research
program of reformulating almost sure statements into respective
statements about algorithmically random points, i.e., algorithmically
random infinite sequences. Any plausible class of random points is of
measure one, see \cite{DowneyHirschfeldt10}, and the effective
versions of theorems substitute phrase ``almost surely'' with ``on all
algorithmically random points''. Usually, randomness in the
Martin-L\"of sense is the desired goal \cite{MartinLof66}. In many
cases, the standard proofs are already constructive, effectivization
of some theorems asks for developing new proofs, but sometimes the
effective versions are false.

In this paper, we will successfully show that the algorithmic
randomness theory is mature enough to make the theory of universal
coding and prediction for stationary ergodic sources effective in the
Martin-L\"of sense. The main keys to this success are: the framework
for randomness with respect to uncomputable measures by Reimann and
Slaman \cite{Reimann08,ReimannSlaman15}, the effective Birkhoff
ergodic theorem
\cite{vyugin1998ergodic,BienvenuOthers12,FranklinOthers12},
generalized by us to the version resembling Breiman's ergodic theorem
\cite{Breiman57}, and an effective Azuma theorem, which follows from a
result of Solovay (unpublished, see
\cite{DowneyHirschfeldt10})---which we call here the effective
Borel-Cantelli lemma---and the Azuma inequality \cite{Azuma67}. As a
little surprise there is also a negative result concerning universal
forward estimators---Theorem \ref{theoSpoiltInduced}. Not everything
can be made effective.

The organization of the paper is as follows. In Section
\ref{secPreliminaries}, we discuss preliminaries: notation (Subsection
\ref{secNotation}), stationary and ergodic measures (Subsection
\ref{secErgodic}), algorithmic randomness (Subsection
\ref{secRandomness}), and some known effectivizations (Subsection
\ref{secKnown}). Section \ref{secMain} contains main results
concerning: universal coding (Subsection \ref{secCoding}), universal
prediction (Subsection \ref{secPrediction}), universal predictors
induced by universal backward estimators (Subsection \ref{secInduced})
and by universal codes (Subsection \ref{secInducedII}), as well as the
PPM measure (Subsection \ref{secPPM}), which constitutes a simple
example of a universal code and a universal predictor.

\section{Preliminaries}
\label{secPreliminaries}

In this section we familiarize the readers with our notation, we
recall the concepts of stationary and ergodic measures, we discuss
various sorts of algorithmic randomness, and we recall known facts
from the effectivization program.

\subsection{Notation}
\label{secNotation}

Throughout this paper, we consider the standard measurable space
$(\mathbb{X}^{\mathbb{Z}},\mathcal{X}^{\mathbb{Z}})$ of two-sided
infinite sequences over a finite alphabet
$\mathbb{X}=\klam{a_1,..,a_D}$, where $D\ge 2$.  (Occasionaly, we also
apply the space of one-sided infinite sequences
$(\mathbb{X}^{\mathbb{N}},\mathcal{X}^{\mathbb{N}})$.)  The points of
the space are (infinite) sequences
$x=(x_i)_{i\in\mathbb{Z}}\in\mathbb{X}^{\mathbb{Z}}$. We also denote
(finite) strings $x_j^k=(x_i)_{j\le i\le k}$, where
$x_j^{j-1}=\lambda$ equals the empty string. By
$\mathbb{X}^*=\bigcup_{n\ge 0}\mathbb{X}^n$ we denote the set of
strings of an arbitrary length including the singleton
$\mathbb{X}^0=\klam{\lambda}$. We use random variables
$X_k((x_i)_{i\in\mathbb{Z}}):=x_k$. Having these, the $\sigma$-field
$\mathcal{X}^{\mathbb{Z}}$ is generated by cylinder sets
$(X_{-|\sigma|+1}^{|\tau|}=\sigma\tau)$ for all
$\sigma,\tau\in\mathbb{X}^*$.  We tacitly assume that $P$ and $R$
denote probability measures on
$(\mathbb{X}^{\mathbb{Z}},\mathcal{X}^{\mathbb{Z}})$. For any
probability measure $P$, we use the shorthand notations
$P(x_1^n):=P(X_1^n=x_1^n)$ and
$P(x_j^n|x_1^{j-1}):=P(X_j^n=x_j^n|X_1^{j-1}=x_1^{j-1})$.  Notation
$\log x$ denotes the binary logarithm, whereas $\ln x$ is the natural
logarithm.

\subsection{Stationary and ergodic measures}
\label{secErgodic}

Let us denote the measurable shift operation
$T((x_i)_{i\in\mathbb{Z}}):=(x_{i+1})_{i\in\mathbb{Z}}$ for two-sided
infinite sequences $(x_i)_{i\in\mathbb{Z}}\in\mathbb{X}^{\mathbb{Z}}$.
\begin{Definition}[stationary measures]
  A probability measure $P$ on
  $(\mathbb{X}^{\mathbb{Z}},\mathcal{X}^{\mathbb{Z}})$ is called
  stationary if $P(T^{-1}(A))=P(A)$ for all events
  $A\in\mathcal{X}^{\mathbb{Z}}$.
\end{Definition}
\begin{Definition}[ergodic measures]
  A probability measure $P$ on
  $(\mathbb{X}^{\mathbb{Z}},\mathcal{X}^{\mathbb{Z}})$ is called
  ergodic if for each event $A\in\mathcal{X}^{\mathbb{Z}}$ such that
  $T^{-1}(A)=A$ we have either $P(A)=1$ or $P(A)=0$.
\end{Definition}
The class of stationary ergodic probability measures has various nice
properties guaranteed by the collection of fundamental results called
ergodic theorems. Typically stationary ergodic measures are not
computable (e.g., consider independent biased coin tosses with a
common uncomputable bias) but they allow for computable universal
coding and prediction on algorithmically random points, as it will be
explained in Section \ref{secMain}.

\subsection{Sorts of randomness}
\label{secRandomness}

Now let us discuss some computability notions. In the following,
\textit{computably enumerable} is abbreviated as \textit{c.e.} Given a
real $r$, the set $\klam{q\in\mathbb{Q}:q<r}$ is called the left cut
of $r$. A real function $f$ with arguments in a countable set is
called computable or left-c.e.\ respectively if the left cuts of
$f(\sigma)$ are uniformly computable or c.e.\ given an enumeration of
$\sigma$. For an infinite sequence $s\in\mathbb{X}^{\mathbb{Z}}$, we
say that real functions $f$ are $s$-computable or $s$-left-c.e.\ if
they are computable or left-c.e.\ with oracle $s$. Similarly, for
a real function $f$ taking arguments in $\mathbb{X}^{\mathbb{Z}}$, we
will say that $f$ is $s$-computable or $s$-left-c.e.\ if left cuts of
$f(x)$ are uniformly computable or c.e.\ with oracles
$x\oplus s:=(...,x_{-1},s_{-1},x_0,s_0,x_1,s_1,...)$. This induces
in effect $s$-computable and $s$-left-c.e.\ random variables on
$(\mathbb{X}^{\mathbb{Z}},\mathcal{X}^{\mathbb{Z}})$.

For stationary ergodic measures, we need a definition of
algorithmically random points with respect to an arbitrary, i.e.,
not necessarily computable probability measure on
$(\mathbb{X}^{\mathbb{Z}},\mathcal{X}^{\mathbb{Z}})$.  A simple
definition thereof was proposed by Reimann \cite{Reimann08} and
Reimann and Slaman \cite{ReimannSlaman15}. This definition is
equivalent to earlier approaches by Levin
\cite{Levin73,Levin76,Levin84} and G\'acs \cite {Gacs05} as shown by
Day and Miller \cite{DayMiller13} and we will use it since it leads to
straightforward generalizations of the results in Section
\ref{secKnown}. The definition is based on measure representations.
Let $\mathcal{P}(\mathbb{X}^{\mathbb{Z}})$ be the space of probability
measures on $(\mathbb{X}^{\mathbb{Z}},\mathcal{X}^{\mathbb{Z}})$. A
measure $P\in\mathcal{P}(\mathbb{X}^{\mathbb{Z}})$ is called
$s$-computable if real function
$(\sigma,\tau)\mapsto P(X_{-|\sigma|+1}^{|\tau|}=\sigma\tau)$ is
$s$-computable. Similarly, a representation function is a function
$\rho:\mathbb{X}^{\mathbb{Z}}\rightarrow\mathcal{P}(\mathbb{X}^{\mathbb{Z}})$
such that real function
$(\sigma,\tau,s)\mapsto \rho(s)(X_{-|\sigma|+1}^{|\tau|}=\sigma\tau)$
is computable.  Subsequently, we say that an infinite sequence
$s\in\mathbb{X}^{\mathbb{Z}}$ is a representation of measure $P$ if
there exists a representation function $\rho$ such that
$\rho(s)=P$. We note that any measure $P$ is $s$-computable for any
representation $s$ of $P$.

We will consider two important sorts of algorithmically random points:
Martin-L\"of or 1-random points and weakly 2-random points with
respect to an arbitrary stationary ergodic measure $P$ on
$(\mathbb{X}^{\mathbb{Z}},\mathcal{X}^{\mathbb{Z}})$.  Note that the
following notions are typically defined for one-sided infinite
sequences over the binary alphabet and computable measures $P$. In the
following parts of this paper, let an infinite sequence
$s\in\mathbb{X}^{\mathbb{Z}}$ be a representation of measure $P$.
\begin{Definition}
  A collection of events
  $U_1,U_2,\ldots\in\mathcal{X}^{\mathbb{Z}}$ is called uniformly
  $s$-c.e.\ if and only if there is a collection of sets
  $V_1,V_2,\ldots\subset\mathbb{X}^*\times\mathbb{X}^*$ such that
  $$U_i=\klam{x\in\mathbb{X}^{\mathbb{Z}}:\exists (\sigma,\tau)\in V_i:
    x_{-|\sigma|+1}^{|\tau|}=\sigma\tau}$$ and sets $V_1,V_2,\ldots$
  are uniformly $s$-c.e.
\end{Definition}
\begin{Definition}[Martin-L\"of test]
  A uniformly $s$-c.e.\ collection of events
  $U_1,U_2,\ldots\in\mathcal{X}^{\mathbb{Z}}$ is called a
  Martin-L\"of $(s,P)$-test if $P(U_n)\leq 2^{-n}$ for every
  $n\in\mathbb{N}$.
\end{Definition}
\begin{Definition}[Martin-L\"of or 1-randomness] 
  A point $x\in\mathbb{X}^{\mathbb{Z}}$ is called Martin-L\"of
  $(s,P)$-random or $1$-$(s,P)$-random if for each Martin-L\"of
  $(s,P)$-test $U_1,U_2,\ldots$ we have $x\not\in\bigcap_{i\ge 1}
  U_i$. A point is called Martin-L\"of $P$-random or $1$-$P$-random if
  it is $1$-$(s,P)$-random for some representation $s$ of $P$.
\end{Definition}

Subsequently, an event $C\in\mathcal{X}^{\mathbb{Z}}$ is called a
$\Sigma^0_2(s)$ event if there exists a uniformly $s$-c.e.\ sequence
of events $U_1,U_2,\ldots$ such that
$\mathbb{X}^{\mathbb{Z}}\setminus C=\bigcap_{i\ge 1} U_i$.
\begin{Definition}[weak $2$-randomness]
  A point $x\in\mathbb{X}^{\mathbb{Z}}$ is called weakly
  $2$-$(s,P)$-random if $x$ is contained in every $\Sigma^0_2(s)$
  event $C$ such that $P(C)=1$. A point is called weakly
  $2$-$P$-random if it is weakly $2$-$(s,P)$-random for some
  representation $s$ of $P$.
\end{Definition}
The sets of weakly $2$-random points are strictly smaller than the
respective sets of $1$-random points, see
\cite{DowneyHirschfeldt10}.

In general, there is a whole hierarchy of algorithmically random
points, such as (weakly) $n$-random points, where $n$ runs over
natural numbers. For our purposes, however, only $1$-random points and
weakly $2$-random points matter since the following proposition sets
the baseline for effectivization:
\begin{Proposition}[folklore]\label{theoEffectivization}
  Let $Y_1,Y_2,\ldots$ be a sequence of uniformly $s$-computable
  random variables. If limit $\lim_{n\to\infty}Y_n$ exists $P$-almost
  surely, then it exists on all weakly $2$-$(s,P)$-random points.
\end{Proposition}
The above proposition is obvious since the set of points on which
limit $\lim_{n\to\infty}Y_n$ exists is a $\Sigma^0_2(s)$ event.  The
effectivization program aims to strengthen the above claim to
$1$-$P$-random points (or even weaker notions such as Schnorr
randomness) but this need not always be feasible.
In particular, one can observe that:
\begin{Proposition}[folklore]\label{theoDelta02}
  Let $P$ be a non-atomic computable measure on
  $\mathbb{X}^{\mathbb{N}}$.  Then there exists a computable function
  $f:\mathbb{X}^*\rightarrow\{0,1\}$ such that the limit
  $\lim_{n\to\infty}f(X_1^n)$ exists and is equal zero $P$-almost
  surely but it is not defined on exactly one point, which is
  $1$-$P$-random.
\end{Proposition}
This fact is a simple consequence of the existence of $\Delta^0_2$
$1$-$P$-random sequences (for a computable $P$) and may be also
interpreted in terms of learning theory (cf.\
\cite{OshersonWeinstein08} and the upcoming paper
\cite{SteiferLearning}).

\subsection{Known effectivizations}
\label{secKnown}

Many probabilistic theorems have been effectivized so far.  Usually
they were stated for computable measures but their generalizations for
uncomputable measures follow easily by relativization, i.e., putting a
representation $s$ of measure $P$ into the oracle. In this section, we
list several known effectivizations of almost sure theorems which we
will use further.

As shown by Solovay (unpublished, see \cite{DowneyHirschfeldt10}), we
have this effective version of the Borel-Cantelli lemma:
\begin{Proposition}[effective Borel-Cantelli lemma]
  \label{theoBorelCantelli} Let $P$ be a probability measure.  If a
  uniformly $s$-c.e.\ sequence of events
  $U_0,U_1,\ldots\in\mathcal{X}^{\mathbb{Z}}$ satisfies
  $\sum_{i=1}^\infty P(U_n)<\infty$ then
  $\sum_{i=1}^\infty \boole{x\in U_n}<\infty$ on each
  $1$-$(s,P)$-random point $x$.
\end{Proposition}

By the effective Borel-Cantelli lemma, Proposition
\ref{theoBorelCantelli}, follows the effective version of the Barron
lemma \cite[Theorem 3.1]{Barron85b}:
\begin{Proposition}[effective Barron lemma]
  \label{theoBarron}
  For any probability measure $P$ and any $s$-computable probability
  measure $R$, on $1$-$(s,P)$-random points we have
  \begin{align}
    \lim_{n\to\infty} \kwad{-\log R(X_1^n)+\log P(X_1^n)+2\log n}=\infty.
  \end{align}
\end{Proposition}

Random variables $P(x_0|X_{-n}^{-1})$ for $n\ge 1$ form a martingale
process. Thus applying the effective martingale convergence
\cite{takahashi2008definition}, we obtain the effective L\'evy law in
particular:
\begin{Proposition}[effective L\'evy law]
  \label{theoLevy}
  For a stationary probability measure $P$, on $1$-$P$-random
  points there exist limits
    \begin{align}
      P(x_0|X_{-\infty}^{-1}):=\lim_{n\to\infty} P(x_0|X_{-n}^{-1}).
    \end{align}
\end{Proposition}

Some celebrated result is the effective Birkhoff ergodic theorem
\cite{vyugin1998ergodic,BienvenuOthers12,FranklinOthers12}. In the
following, $\mean X:=\int X dP$ stands for the expectation of a random
variable $X$ with respect to measure $P$.
\begin{Proposition}[effective Birkhoff ergodic theorem
  \mbox{\cite[Theorem 10]{BienvenuOthers12}}]
    \label{theoBirkhoff}
    For a stationary ergodic probability measure $P$ and an
    $s$-left-c.e.\ real random variable $G$ such that $G\ge 0$ and
    $\mean G<\infty$, on $1$-$(s,P)$-random points we have
  \begin{align}
    \label{Birkhoff}
    \lim_{n\to\infty}\frac{1}{n}\sum_{i=0}^{n-1} G\circ T^i
    =
    \mean G.
  \end{align}
\end{Proposition}

The proof of the next proposition is an easy application of
Proposition \ref{theoBirkhoff} and properties of left-c.e.\ functions.
\begin{Proposition}[effective Breiman ergodic theorem
  \cite{steifer2020thesis}]
  \label{theoBreiman}
  For a stationary ergodic probability measure $P$ and uniformly
  $s$-computable real random variables $(G_i)_{i\ge 0}$ such that
  $G_n\ge 0$, $\mean \sup_n G_n<\infty$, and limit $\lim_{n\to\infty} G_n$
  exists $P$-almost surely, on $1$-$(s,P)$-random points we have
  \begin{align}
    \label{Breiman}
    \lim_{n\to\infty}\frac{1}{n}\sum_{i=0}^{n-1} G_i\circ T^i
    =
    \mean\lim_{n\to\infty} G_n.
  \end{align}
\end{Proposition}
\begin{proof}
  Let $H_k:=\sup_{t>k}G_t\ge 0$. Then $G_t\le H_k$ for all $t>k$ and
  consequently,
  \begin{align}
    \limsup_{n\to\infty}\frac{1}{n}\sum_{i=0}^{n-1}G_i\circ T^i\le
    \limsup_{n\to\infty}\frac{1}{n}\sum_{i=0}^{n-1}H_k\circ T^i. 
  \end{align}
  Observe that the supremum $H_k$ of uniformly $s$-computable
  functions $G_{k+1},G_{k+2},\ldots$ is $s$-left-c.e. Indeed, to
  enumerate the left cut of the supremum $H_k(x)$ we simultaneously
  enumerate the left cuts of $G_{k+1}(x),G_{k+2}(x),\ldots$. This is
  possible since every $s$-computable function is also
  $s$-left-c.e. Moreover, we are considering only countably many
  functions, hence we can guarantee that an element of each left cut
  appears in the enumeration at least once.

  Now, since for all $k\ge 0$ random variables $H_k$ are
  $s$-left-c.e.\ then by Theorem \ref{theoBirkhoff} (effective
  Birkhoff ergodic theorem), on $1$-$(s,P)$-random points we have
  \begin{align}
    \limsup_{n\to\infty}\frac{1}{n}\sum_{i=0}^{n-1}G_i\circ T^i\le
    \mean H_k.
  \end{align}
  Since $H_k\ge 0$ and $\mean \sup_k H_k<\infty$ then by the dominated
  convergence,
  \begin{align}
    \inf_{k\ge 0} \mean H_k=
    \lim_{k\to\infty}\mean H_k=
    \mean\lim_{k\to\infty}H_k=
    \mean\lim_{n\to\infty}G_n.   
  \end{align}
  Thus,
  \begin{align}
    \label{BreimanUpper}
    \limsup_{n\to\infty}\frac{1}{n}\sum_{i=0}^{n-1}G_i\circ T^i\le
    \mean\lim_{n\to\infty}G_n.
  \end{align}
  
  For the converse inequality, consider a natural number $M$ and put
  random variables $\bar H_k:=M-\inf_{t>k}\min\klam{G_t,M}\in
  [0,M]$. We observe that $\bar H_k$ are also $s$-left-c.e.\ since
  $G_t$ are uniformly $s$-computable by the hypothesis. By Theorem
  \ref{theoBirkhoff} (effective Birkhoff ergodic theorem), on
  $1$-$(s,P)$-random points we have
  \begin{align}
    M-\liminf_{n\to\infty}\frac{1}{n}\sum_{i=0}^{n-1}\min\klam{G_i,M}\circ T^i\le
    \mean \bar H_k.
  \end{align}
  Since $0\le \bar H_k\le M$ then by the dominated convergence,
  \begin{align}
    \inf_{k\ge 0} \mean \bar H_k=
    \lim_{k\to\infty}\mean \bar H_k=
    \mean\lim_{k\to\infty}\bar H_k=
    M-\mean\min\klam{\lim_{n\to\infty}G_n,M}.   
  \end{align}
  Hence, regrouping the terms we obtain
  \begin{align}
    \liminf_{n\to\infty}\frac{1}{n}\sum_{i=0}^{n-1}G_i\circ T^i
    &\ge
    \liminf_{n\to\infty}\frac{1}{n}\sum_{i=0}^{n-1}\min\klam{G_i,M}\circ
      T^i
      \nonumber\\
    &\ge
    \mean\min\klam{\lim_{n\to\infty}G_n,M}
    \xrightarrow[M\to\infty]{}  \mean\lim_{n\to\infty}G_n,
    \label{BreimanLower}
  \end{align}
  where the last transition follows by the monotone convergence.  By
  (\ref{BreimanUpper}) and (\ref{BreimanLower}) we derive the claim.
\end{proof}
The almost sure versions of Propositions \ref{theoBirkhoff} and
\ref{theoBreiman} concern random variables which need not be
nonnegative \cite{Breiman57}.

An important result for universal prediction is the Azuma inequality
\cite{Azuma67}, whose following corollary will be used Subsections
\ref{secPrediction} and \ref{secInducedII}.
\begin{Theorem}[effective Azuma theorem]
  \label{theoAzuma}
  For a probability measure $P$ and uniformly $s$-computable real
  random variables $(Z_n)_{n\ge 1}$ such that
  $\abs{Z_n}\le \epsilon_n\sqrt{n/\ln n}$ with
  $\lim_{n\to\infty} \epsilon_n=0$, on $1$-$(s,P)$-random points we have
  \begin{align}
    \label{Azuma}
    \lim_{n\to\infty}\frac{1}{n}\sum_{i=1}^{n}
    \kwad{Z_i-\mean\okra{Z_i\middle|X_1^{i-1}}}
    =
    0.
  \end{align}
\end{Theorem}
\begin{proof}
  Define
  \begin{align}
    Y_n:=\sum_{i=1}^{n}
    \kwad{Z_i-\mean\okra{Z_i\middle|X_1^{i-1}}}.
  \end{align}
  Process $(Y_n)_{n\ge 1}$ is a martingale with respect to process
  $(X_n)_{n\ge 1}$ with increments bounded by inequality
  \begin{align}
    \abs{Z_n-\mean\okra{Z_n\middle|X_1^{n-1}}}\le c_n
    := 2\epsilon_n\sqrt{n/\ln n}.
  \end{align}
  By the Azuma inequality \cite{Azuma67} for any $\epsilon>0$ we
  obtain
  \begin{align}
    P(\abs{Y_n}\ge n\epsilon)
    &\le
    2\exp\okra{-\frac{\epsilon^2n^2}{2\sum_{i=1}^{n} c_i^2}}
    =
      n^{-\alpha_n}
      ,
  \end{align}
  where
  \begin{align}
    \alpha_n&:=\frac{\epsilon^2n}{8\sum_{i=1}^{n} \epsilon_i^2}.
  \end{align}
  Since $\alpha_n\to\infty$, we have
  $\sum_{n=1}^\infty P(\abs{Y_n}\ge n\epsilon)<\infty$ and by
  Proposition \ref{theoBorelCantelli} (effective Borel-Cantelli
  lemma), we obtain (\ref{Azuma}) on $1$-$(s,P)$-random points.
\end{proof}

\section{Main results}
\label{secMain}

This section contains results concerning effective universal coding
and prediction, predictors induced by universal measures and some
examples of universal measures and universal predictors.

\subsection{Universal coding}
\label{secCoding}

Let us begin our considerations with the problem of universal
measures, which is related to the problem of universal coding.
Suppose that we want to compress losslessly a typical sequence
generated by a stationary probability measure $P$. We can reasonably
ask what is the lower limit of such a compression, i.e., what is the
miminal ratio of the encoded string length divided by the original
string length. In information theory, it is well known that the
greatest lower bound of such ratios is given by the entropy rate of
measure $P$.  For a stationary probability measure $P$, we denote its
entropy rate as
\begin{align}
  h_{P}
  &:=\lim_{n\to\infty}\frac{1}{n}\mean\kwad{-\log P(X_1^n)}
  =\lim_{k\to\infty}\mean\kwad{-\log P(X_{k+1}|X_1^k)}
  .
\end{align}
The entropy rate has the interpretation of the minimal asymptotic rate
of lossless encoding of sequences emitted by measure $P$ in various
senses: in expectation, almost surely, or on algorithmically random
points, where the last interpretation will be pursued in this
subsection.

To furnish some theoretical background for universal coding let us
recall the Kraft inequality $\sum_{w\in A} 2^{-\abs{w}}\le 1$, which
holds for any prefix-free subset of strings
$A\subset\klam{0,1}^*$. The Kraft inequality implies in particular
that lossless compresion procedures, called uniquely decodable codes,
can be mapped one-to-one to semi-measures. In particular, if we are
seeking for a universal code, i.e., a uniquely decodable code
$w\mapsto C(w)\in\klam{0,1}^*$ which is optimal for some class of
stochastic sources $P$, we can equivalently seek for a universal
semi-measure of form $w\mapsto R(w):=2^{-\abs{C(w)}}$. Consequently,
the problem of universal coding would be solved if we point out such a
semi-measure $R$ that
\begin{align}
  \lim_{n\to\infty}\frac{1}{n}\kwad{-\log R(X_1^n)}=
   \lim_{n\to\infty}\frac{1}{n}\abs{C(X_1^n)}=h_{P}
\end{align}
for some points that are typical of $P$.

As it is well established in information theory, some initial insight
into the problem of universal coding or universal measures is given by
the Shannon-McMillan-Breiman (SMB) theorem, which states that function
$\frac{1}{n}\kwad{-\log P(X_1^n)}$ tends $P$-almost surely to the
entropy rate $h_{P}$. The classical proofs of this result were given
by Algoet and Cover \cite{AlgoetCover88} and Chung \cite{Chung61}.
An effective version of the SMB theorem was presented by Hochman
\cite{hochman2009upcrossing} and Hoyrup \cite{hoyrup2011dimension}.
\begin{Theorem}[effective SMB theorem
  \cite{hochman2009upcrossing,hoyrup2011dimension}]
  \label{theoSMB}
  For a stationary ergodic probability measure $P$, on
  $1$-$P$-random points we have
 \begin{align}
    \label{SMB}
   \lim_{n\to\infty}\frac{1}{n}\kwad{-\log P(X_1^n)}=h_{P}.
 \end{align}
\end{Theorem}

The essential idea of Hoyrup's proof, which is a bit more complicated,
can be retold using tools developed in Subsection
\ref{secKnown}. Observe first that we have
\begin{align}
  \frac{1}{n}\kwad{-\log P(X_1^n)}=
  \frac{1}{n}\sum_{i=1}^n\kwad{-\log P(X_i|X_{1}^{i-1.})}.
\end{align}
Moreover, we have the uniform bound
\begin{align}
  \mean \sup_{n\ge 0}\kwad{-\log P(X_0|X_{-n}^{-1})}
  \le \mean \kwad{-\log P(X_0)}+\log e
  \le \log e D
  <\infty,
\end{align}
see \cite[Lemma 4.26]{Smorodinsky71}---invoked by Hoyrup as well.
Consequently, the effective SMB theorem follows by Proposition
\ref{theoBreiman} (effective Breiman ergodic theorem) and Proposition
\ref{theoLevy} (effective L\'evy law). In contrast, the reasoning by
Hoyrup was more casuistic---he did not mention the Breimain ergodic
theorem and invoked some more specific statements.

We note in passing that it could be also interesting to check whether
one can effectivize the textbook sandwich proof of the SMB theorem by
Algoet and Cover \cite{AlgoetCover88} using the decomposition of
conditionally algorithmically random sequences by Takahashi
\cite{takahashi2008definition}. However, this step would require some
novel theoretical considerations about conditional algorithmic
randomness for uncomputable measures. We mention this only to point
out a possible direction for future research.

As a direct consequence of the effective SMB theorem and Proposition
\ref{theoBarron} (effective Barron lemma), we obtain this
effectivization of another well-known almost sure statement:
\begin{Theorem}[effective source coding]
  \label{theoCode}
  For any stationary ergodic measure $P$ and any $s$-computable
  probability measure $R$, on $1$-$(s,P)$-random points we have
\begin{align}
  \label{Code}
  \liminf_{n\to\infty}\frac{1}{n}\kwad{-\log R(X_1^n)}
  \ge
  h_{P}
  .
\end{align}  
\end{Theorem}
In the almost sure setting, relationship (\ref{Code}) holds $P$-almost
surely for any stationary ergodic measure $P$ and any (not necessarily
computable) probability measure $R$.

Now we can define universal measures.
\begin{Definition}[universal measure]
  \label{defiUniversalMeasure}
  A computable (not necessarily stationary) probability measure $R$
  is called (weakly) $n$-universal if for any stationary ergodic
  probability measure $P$, on (weakly) $n$-$P$-random points we have
  \begin{align}
    \label{UniversalCode}
    \lim_{n\to\infty}\frac{1}{n}\kwad{-\log R(X_1^n)}= h_{P}
    .
  \end{align}
\end{Definition}
In the almost sure setting, we say that a probability measure $R$ is
almost surely universal if (\ref{UniversalCode}) holds $P$-almost
surely for any stationary ergodic probability measure $P$.  By
Proposition \ref{theoEffectivization}, there are only two practically
interesting cases of computable universal measures: weakly
$2$-universal ones and $1$-universal ones, since every computable
almost surely universal probability measure is automatically weakly
$2$-universal. We stress that we impose computability of (weakly)
$n$-universal measures by definition since it simplifies statements of
some theorems. This should be contrasted with universal prediction \`a
la Solomonoff for left-c.e.\ semimeasures where the universal element
belongs to the class and is not computable \cite{Solomonoff64}.

Computable almost surely universal measures exist if the alphabet
$\mathbb{X}$ is finite.  An important example of an almost surely
universal and, as we will see in Section \ref{secPPM}, also
$1$-universal measure is the Prediction by Partial Matching (PPM)
measure \cite{ClearyWitten84,Ryabko88en2,Ryabko08}.  As we have
mentioned, universal measures are closely related to the problem of
universal coding (data compression) and more examples of universal
measures can be constructed from universal codes, for instance given
in \cite{ZivLempel77,KiefferYang00,CharikarOthers05,Debowski11b},
using the normalization by Ryabko \cite{Ryabko09}. This normalization
is not completely straightforward, since we need to forge
semi-measures into probability measures.

\subsection{Universal prediction}
\label{secPrediction}

Universal prediction is a problem similar to universal coding. In this
problem, we also seek for a single procedure that would be optimal
within a class of probabilistic sources but we apply a different loss
function, namely, we impose the error rate being the density of
incorrect guesses of the next output given previous ones. In spite of
this difference, we will try to state the problem of universal
prediction analogously to universal coding.  A predictor is an
arbitrary total function $f:\mathbb{X}^*\rightarrow\mathbb{X}$. The
predictor induced by a probability measure $P$ will be defined as
\begin{align}
  \label{MetaPredictor}
  f_{P}(x_1^n):=\argmax_{x_{n+1}\in\mathbb{X}} P(x_{n+1}|x_1^n),
\end{align}
where
$\argmax_{x\in \mathbb{X}} g(x):=\min\klam{a\in \mathbb{X}: g(a)\ge
  g(x) \text{ for all $x\in\mathbb{X}$}}$ for the total order
$a_1<....<a_D$ on $\mathbb{X}=\klam{a_1,..,a_D}$.  Moreover, for a
stationary measure $P$, we define the unpredictability rate
\begin{align}
  u_{P}:=\lim_{n\to\infty}\mean\kwad{1-\max_{x_0\in\mathbb{X}} P(x_0|X_{-n}^{-1})}
  .
\end{align}

It is natural to ask whether the unpredictability rate can be related
to entropy rate. Using the Fano inequality \cite{Fano61}, a classical
result of information theory, and its converse \cite{Debowski21}, both
independently brought to algorithmic randomness theory by Fortnow and
Lutz \cite{fortnow2005prediction}, yields this bound:
\begin{Theorem}
  For a stationary measure $P$ over a $D$-element alphabet,
  \begin{align}
    \frac{D}{D-1}\eta\okra{\frac{1}{D}} u_P
    \le
    h_P
    \le
    \eta(u_P)+u_P\log(D-1)
    ,
  \end{align}
  where $\eta(p):=-p\log p-(1-p)\log(1-p)$.
\end{Theorem}
\noindent
Moreover, Fortnow and Lutz \cite{fortnow2005prediction} found out some
stronger inequalities, sandwich-bounding the unpredictability of an
arbitrary sequence in terms of its effective dimension.  The effective
dimension turns out to be a generalization of the entropy rate to
arbitrary sequences \cite{hoyrup2011dimension}, which are not
necessarily random with respect to stationary ergodic measures.

In the less general framework of stationary ergodic measures, using
the Azuma theorem, we can show that no predictor can beat the induced
predictor and the error rate committed by the latter equals the
unpredictability rate $u_{P}$.  The following proposition concerning
the error rates effectivizes the well-known almost sure proposition
(the proof in the almost sure setting is available in
\cite{Algoet94}).
\begin{Theorem}[effective source prediction]
  \label{theoPredictor}
  For any stationary ergodic measure $P$ and any $s$-computable
  predictor $f$, on $1$-$(s,P)$-random points we have
\begin{align}
  \liminf_{n\to\infty}\frac{1}{n}\sum_{i=0}^{n-1}\boole{X_{i+1}\neq f(X_1^i)}
  &\ge
  u_{P}
  .
  \label{Predictor}
\end{align}
Moreover, if the induced predictor $f_{P}$ is $s$-computable then
(\ref{Predictor}) holds with the equality for $f=f_{P}$.
\end{Theorem}
\begin{proof}
  Let measure $P$ be stationary ergodic.  In view of Theorem
  \ref{theoAzuma} (effective Azuma theorem), for any $s$-computable
  predictor $f$, on $1$-$(s,P)$-random points we have
  \begin{align}
    \label{AzumaPredictor}
    \lim_{n\to\infty}\frac{1}{n}\sum_{i=0}^{n-1}
    \kwad{\boole{X_{i+1}\neq f(X_1^i)}-P(X_{i+1}\neq f(X_1^i)|X_1^i)}=0.
  \end{align}
 Moreover, we have
 \begin{align}
   \label{ErrorProb}
  P(X_{i+1}\neq f(X_1^i)|X_1^i)\ge
  1-\max_{x_{i+1}\in\mathbb{X}}P(x_{i+1}|X_1^i)
\end{align}
Subsequently, we observe that limits
$\lim_{n\to\infty} P(x_0|X_{-n}^{-1})$ exist on $1$-$(s,P)$-random
points by Proposition \ref{theoLevy} (effective L\'evy law). Thus by
Proposition \ref{theoBreiman} (effective Breiman ergodic theorem) and
the dominated convergence, on $1$-$(s,P)$-random points we obtain
\begin{align}
  \lim_{n\to\infty}\frac{1}{n}\sum_{i=0}^{n-1}
    \kwad{1-\max_{x_{i+1}\in\mathbb{X}}P(x_{i+1}|X_1^i)}
  &=
    \mean\lim_{n\to\infty}\kwad{1-\max_{x_0\in\mathbb{X}}
    P(x_0|X_{-n}^{-1})}
    \nonumber\\
  &=
    u_{P}.
    \label{BreimanPredictor}
\end{align}
Hence inequality (\ref{Predictor}) follows by (\ref{AzumaPredictor}),
(\ref{ErrorProb}) and (\ref{BreimanPredictor}). Similarly, the
equality in (\ref{Predictor}) for $f=f_{P}$ follows by noticing that
inequality (\ref{ErrorProb}) turns out to be the equality in this
case.
\end{proof}
In the almost sure setting, relationship (\ref{Predictor}) holds
$P$-almost surely for any stationary ergodic measure $P$ and any (not
necessarily computable) predictor $f$.

We can see that there can be some problem in the effectivization of
relationship (\ref{Predictor}) caused by the induced predictor $f_{P}$
possibly not being $s$-computable for certain representations $s$ of
measure $P$---since sometimes testing the equality of two real numbers
cannot be done in a finite time. However, probabilities
$P(X_{i+1}\neq f_{P}(X_1^i)|X_1^i)$ are always $s$-computable. Thus,
we can try to define universal predictors in the following way.
\begin{Definition}[universal predictor]
  \label{defiUniversalPredictor}
  A computable predictor $f$ is called (weakly) $n$-universal if
  for any stationary ergodic probability measure $P$, on (weakly)
  $n$-$P$-random points we have
\begin{align}
  \label{UniversalPredictor}
  \lim_{n\to\infty}\frac{1}{n}\sum_{i=0}^{n-1}\boole{X_{i+1}\neq f(X_1^i)}
  =
  u_{P}.
\end{align}  
\end{Definition}
In the almost sure setting, we say that a predictor $f$ is almost
surely universal if (\ref{UniversalPredictor}) holds $P$-almost surely
for any stationary ergodic probability measure $P$.  Almost surely
universal predictors exist if the alphabet $\mathbb{X}$ is finite
\cite{Bailey76,Ornstein78,Algoet94,GyorfiLugosiMorvai99,GyorfiLugosi01}.
In \cite{steifer2020thesis} it was proved that the almost sure
predictor by \cite{GyorfiLugosiMorvai99} is also $1$-universal.

\subsection{Predictors induced by backward estimators}
\label{secInduced}

The almost surely universal predictors by
\cite{Bailey76,Ornstein78,Algoet94,GyorfiLugosiMorvai99,GyorfiLugosi01}
were constructed without a reference to universal
measures. Nevertheless, these constructions are all based on
estimation of conditional probabilities. For a stationary ergodic
process one can consider two separate problems: backward and forward
estimation. The first problem is naturally connected to prediction. We
want to estimate the conditional probability of $(n+1)$-th bit given the
first $n$ bits. Is it possible that, as we increase $n$, our estimates
converge to the true value at some point? To be precise, we ask
whether there exists a probability measure $R$ such that for every
stationary ergodic measure $P$ we have $P$-almost surely
\begin{align}
  \lim_{n\to\infty}\sum_{x_{n+1}\in\mathbb{X}}\abs{R(x_{n+1}|X_{1}^{n-1})-P(x_{n+1}|X_{1}^{n-1})}=0.
\end{align}
It was shown by Bailey \cite{Bailey76} that this is not possible. As
we are about to see, we can get something a bit weaker, namely, the
convergence in Cesaro averages. But to get there, it will be helpful
to consider a bit different problem.

Suppose again that we want to estimate a conditional probability but
the bit that we are interested in is fixed and we are looking more and
more into the past.  In this scenario, we want to estimate the
conditional probability $P(x_0|X_{-\infty}^{-1})$ and we ask whether
increasing the knowledge of the past can help us achieve the perfect
guess. Precisely, we ask if there exists a probability measure $R$
such that for every stationary ergodic measure $P$ we have 
$P$-almost surely
\begin{align}
  \lim_{n\to\infty}
  \sum_{x_0\in\mathbb{X}}\abs{R(x_0|X^{-1}_{-n})-P(x_0|X_{-\infty}^{-1})}
  =0.
\end{align}
It was famously shown by Ornstein that such estimators
exist. (Ornstein proved this for binary-valued processes but the
technique can be generalized to finite-valued processes.)
\begin{Theorem}[Ornstein theorem \cite{Ornstein78}]
  \label{theoOrnstein}
  Let the alphabet be finite.  There exists a computable measure $R$
  such that for every stationary ergodic measure $P$ we have
  $P$-almost surely that
  \begin{align}
    \label{Ornstein}
    \lim_{n\to\infty}
    \sum_{x_0\in\mathbb{X}}\abs{R(x_0|X^{-1}_{-n})-P(x_0|X_{-\infty}^{-1})}
    =0.
  \end{align}
\end{Theorem}

\begin{Definition}
  \label{defiUniversalBackward}
  We call a measure $R$ an almost surely universal backward estimator
  when it satisfies condition (\ref{Ornstein}) $P$-almost surely for
  every stationary ergodic measure $P$, whereas it is called a
  (weakly) $n$-universal backward estimator if $R$ is computable and
  convergence (\ref{Ornstein}) holds on all respective (weakly)
  $n$-$P$-random points.
\end{Definition}

One can come up with a naive idea: What if we take a
universal backward estimator and use it in a forward fashion?
Surprisingly, this simple trick gives us almost everything we can get,
i.e., a forward estimator that converges to the conditional
probability on average. Bailey \cite{Bailey76} showed that for an
almost surely universal backward estimator $R$ and for every
stationary ergodic measure $P$ we have $P$-almost surely
\begin{align}
    \label{Bailey}
    \lim_{n\to\infty}\frac{1}{n}\sum_{i=0}^{n-1}\sum_{x_{i+1}\in\mathbb{X}}
    \abs{R(x_{i+1}|X_{1}^{i})-P(x_{i+1}|X_{1}^{i})}=0.
  \end{align}
The proof of this fact is a direct application of the Breiman ergodic
theorem. Since we have a stronger effective version of the Breiman
theorem (Theorem \ref{theoBreiman}), we can strengthen Bailey's result
to an effective version as well.  It turns out that even if we take a
backward estimator that is good only almost surely (possibly failing
on some random points), then the respective result for the forward
estimation will hold in the strong sense---on every $1$-$P$-random
point.
\begin{Theorem}[effective Bailey theorem]
  \label{theoBailey}
  Let $R$ be a computable almost surely universal backward
  estimator. For every stationary ergodic measure $P$ on
  $1$-$P$-random points we have (\ref{Bailey}).
\end{Theorem}
\begin{proof}
  Let $R$ be a computable almost surely universal backward estimator.
  Fix an $x\in\mathbb{X}$. By Proposition \ref{theoLevy} (effective
  L\'evy law), for every stationary ergodic probability measure $P$ we
  have $P$-almost surely
  \begin{align}
    \lim_{n\to\infty}\abs{R(x|X^{-1}_{-n})-P(x|X_{-n}^{-1})}=0.
  \end{align}
  Note that the bound
  $0\le \abs{R(x|X^{-1}_{-n})-P(x|X_{-n}^{-1})}\le 1$ holds uniformly.
  Moreover, variables $R(x|X^{-1}_{-n})-P(x|X_{-n}^{-1})$ are
  uniformly $s$-computable for any representation $s$ of $P$. Hence,
  we can apply Theorem \ref{theoBreiman} (effective Breiman ergodic
  theorem) to obtain
  \begin{align}
    \lim_{n\to\infty}\frac{1}{n}\sum_{i=0}^{n-1}
    \abs{R(x|X_{1}^{i})-P(x|X_{1}^{i})}=\mean 0=0.
  \end{align}
  for $1$-$P$-random points. The claim follows from this immediately.
\end{proof}

\begin{Definition}
  \label{defiUniversalForward}
  We call a measure $R$ an almost surely universal forward estimator
  when it satisfies condition (\ref{Bailey}) $P$-almost surely for
  every stationary ergodic measure $P$, whereas it is called a
  (weakly) $n$-universal forward estimator if $R$ is computable and
  convergence (\ref{Bailey}) holds on all respective (weakly)
  $n$-$P$-random points.
\end{Definition}

One can expect that the predictor $f_R$ induced by a universal forward
estimator $R$ in the sense of Definition \ref{defiUniversalForward} is
also universal in the sense of Definition
\ref{defiUniversalPredictor}.  This is indeed true. To show this fact,
we will first prove a certain inequality for induced predictors, which
generalizes the result from \cite[Theorem 2.2]{DevroyeGyorfiLugosi96}
for binary classifiers. This particular observation seems to be new.
\begin{Proposition}[prediction inequality]
  \label{theoBound}
  Let $p$ and $q$ be two probability distributions over a countable
  alphabet $\mathbb{X}$. For $x_p=\argmax_{x\in\mathbb{X}} p(x)$ and
  $x_q=\argmax_{x\in\mathbb{X}} q(x)$, we have inequality
\begin{align}
  0&\le p(x_p)-p(x_q)\le \sum_{x\in\mathbb{X}} \abs{p(x)-q(x)}.
    \label{Bound}
\end{align}  
\end{Proposition}
\begin{proof}
  Without loss of generality, assume $x_p\neq x_q$.  By the definition
  of $x_p$ and $x_q$, we have $p(x_p)-p(x_q)\ge 0$ and
  $q(x_q)-q(x_p)\ge 0$.  Hence we obtain
  \begin{align}
    0&\le p(x_p)-p(x_q)
       \le p(x_p)-p(x_q)
       -q(x_p)+q(x_q)
       \nonumber\\
     &\le \abs{p(x_p)-q(x_p)}
       +\abs{p(x_q)-q(x_q)}
       \le \sum_{x}\abs{p(x)-q(x)}.
  \end{align}
\end{proof}

Now we can show a general result about universal predictors induced by
forward estimators of conditional probabilities.
\begin{Theorem}[effective induced prediction I]
  \label{theoBaileyInduced}
  For a $1$-universal forward estimator $R$, the induced predictor
  $f_R$ is $1$-universal if $f_R$ is computable.
\end{Theorem}
\begin{proof}
  Let $R$ be $1$-universal forward estimator. By the definition, for
  every stationary ergodic measure $P$ and all $1$-$P$-random points
\begin{align}
    \lim_{n\to\infty}
    \frac{1}{n}\sum_{i=0}^{n-1}\sum_{x_{i+1}\in\mathbb{X}}
    \abs{P(x_{i+1}|X_1^i)-R(x_{i+1}|X_1^i)}
    = 0.
\end{align}
Consequently, combining this with Proposition \ref{theoBound}
(prediction inequality) yields on $1$-$P$-random points
\begin{align}
  \lim_{n\to\infty}
  \frac{1}{n}\sum_{i=0}^{n-1}
  \kwad{P(X_{i+1}\neq f_R(X_1^i)|X_1^i)-
  P(X_{i+1}\neq f_{P}(X_1^i)|X_1^i)}
  =
  0
  .
\end{align}
Now, we notice that by (\ref{AzumaPredictor}), we have on
$1$-$P$-random points
\begin{align}
   \lim_{n\to\infty}
  \frac{1}{n}\sum_{i=0}^{n-1}
  \kwad{\boole{X_{i+1}\neq f_R(X_1^i)}-P(X_{i+1}\neq
  f_R(X_1^i)|X_1^i)}
  &=
  0
  ,
  \\
  \lim_{n\to\infty}
  \frac{1}{n}\sum_{i=0}^{n-1}
  \kwad{\boole{X_{i+1}\neq f_P(X_1^i)}-P(X_{i+1}\neq
  f_P(X_1^i)|X_1^i)}
  &=
  0
  .
\end{align}
Combining the three above observations completes the proof.
\end{proof}

Interestingly, it suffices for a measure to be a computable almost
surely universal backward estimator to yield a $1$-universal forward
estimator and, consequently, a $1$-universal predictor. In contrast,
we can easily see that a computable almost surely universal forward
estimator does not necessarily induce a $1$-universal predictor.
\begin{Theorem}
  \label{theoSpoiltInduced}
  There exists a computable almost surely universal forward estimator
  $R$ such that the induced predictor $f_{R}$ is not $1$-universal.
\end{Theorem}
\begin{proof}
  Let us take $\mathbb{X}=\klam{0,1}$ and restrict ourselves to
  one-sided space $\mathbb{X}^{\mathbb{N}}$ without loss of
  generality.  Fix a computable almost surely universal forward
  estimator $Q$. Let $P_0$ be the computable measure of a
  Bernoulli($\theta$) process, i.e.,
  $P_0(x_1^n)=\prod_{i=1}^n\theta^{x_i}(1-\theta)^{1-x_i}$, where
  $\theta>1/2$ is rational. Observe that by Proposition
  \ref{theoDelta02} there exists a point $y\in\mathbb{X}^{\mathbb{N}}$
  which is $1$-$P_0$-random and a computable function
  $g:\mathbb{X}^*\rightarrow\klam{0,1}$ such that $P_0(A)=0$ and
  $A=\klam{y}$ for event
  \begin{align}
   A:=(\#\klam{i\in\mathbb{N}:g(X_1^i)=1}=\infty). 
  \end{align}
  In other words, there is a computable method to single out some
  $1$-$P_0$-random point $y$ out of the set of sequences
  $\mathbb{X}^{\mathbb{N}}$.  In particular, we can use function $g$
  to spoil measure $Q$ on that point $y$ while preserving the property
  of an almost surely universal forward estimator. We will denote the
  spoilt version of measure $Q$ by $R$. Conditional distributions
  $R(X_{m+1}|X_1^m)$ will differ from $Q(X_{m+1}|X_1^m)$ for
  infinitely many $m$ on point $y$ and for finitely many $m$
  elsewhere.

  Let $K(x_1^n):=\#\klam{i\le n:g(x_1^i)=1}$.  The construction of
  measure $R$ proceeds by induction on the string length together with
  an auxiliary counter $U$.  We let $R(x_1):=Q(x_1)$ and
  $U(x_1):=0$. Suppose that $R(x_1^n)$ and $U(x_1^n)$ are defined but
  $R(x_1^{n+1})$ is not. If $U(x_1^n)\ge K(x_1^n)$ then we put
  $R(x_{n+1}|x_1^n):=Q(x_{n+1}|x_1^n)$ and
  $U(x_1^{n+1}):=U(x_1^n)$. Else, if $U(x_1^n)< K(x_1^n)$ then we put
  $R(x_{n+1}^{n+N}|x_1^n):=\prod_{i=n+1}^{n+N}\theta^{1-x_i}(1-\theta)^{x_i}$
  (reverted compared to the definition of $P_0$!)  and
  $U(x_{n+N}):=K(x_1^n)$ where $N$ is the smallest number such that
  \begin{align}
    \frac{1}{n+N}
    \kwad{\sum^{n-1}_{i=0}
    P_0(X_{i+1}\neq f_R(x_1^{i})|X_1^i=x_1^{i})+N\theta}
    \ge \frac{1}{2}. 
  \end{align}
  Such number $N$ exists since
  $P_0(X_{i+1}\neq f_R(x_1^{i})|X_1^i=x_1^{i})> 1-\theta$. This
  completes the construction of $R$.

  The sets of $1$-$P$-random sequences are disjoint for distinct
  stationary ergodic $P$ by Theorem \ref{theoBirkhoff} (effective Birkhoff ergodic theorem).  Hence
  $K(X_1^n)$ is bounded $P$-almost surely for any stationary ergodic
  $P$. Consequently, since $U(X_1^n)$ is non-decreasing then
  $P$-almost surely there exists a random number $M<\infty$ such that
  for all $m>M$ we have $R(X_{m+1}|X_1^m)=Q(X_{m+1}|X_1^m)$.  Hence
  $R$ inherits the property of an almost surely universal forward
  estimator from $Q$.
  
  Now let us inspect what happens on $y$. Since $K(X_1^n)$ is
  unbounded on $y$ then by the construction of $R$, we obtain on $y$
  that $U(X_1^n)<K(X_1^n)$ holds infinitely often and
  \begin{align}
    \limsup_{n\to\infty}\frac{1}{n}\sum^{n}_{i=0}
    P_0(X_{i+1}\neq f_R(X_1^{i})|X_1^{i})\ge \frac{1}{2}>u_{P_0}=1-\theta.
  \end{align}
  Hence predictor $f_{R}$ is not $1$-universal.
\end{proof}

\subsection{Predictors induced by universal measures}
\label{secInducedII}

Following the work of Ryabko \cite{Ryabko09}, see also
\cite{RyabkoAstolaMalyutov16}, we can ask a natural question whether
predictors induced by some universal measures in the sense of
Definition \ref{defiUniversalMeasure}, such as the PPM measure
\cite{ClearyWitten84,Ryabko88en2,Ryabko08} to be discussed in Section
\ref{secPPM}, are also universal. Ryabko was close to demonstrate the
analogous implication in the almost sure setting but did not provide
the complete proof. He has shown this proposition:
\begin{Theorem}[\cite{Ryabko08}]
  \label{theoRyabko}
  Let $R$ be an almost surely universal measure and $P$ be a
  stationary ergodic measure. We have $P$-almost surely
  \begin{align}
    \label{ExpectedBailey}
    \lim_{n\to\infty}\mean\frac{1}{n}\sum_{i=0}^{n-1}
    \abs{P(X_{i+1}|X_0^i)-R(X_{i+1}|X_0^i)}=0.
  \end{align}
\end{Theorem}

Condition (\ref{ExpectedBailey}) is quite close to condition
(\ref{Bailey}), i.e., universal forward estimator, which---as we have
shown in Theorem \ref{theoBaileyInduced}---implies universality of the
induced predictor. Close does not mean exactly the same. If we would
like to derive universality of the induced predictor directly from
(\ref{ExpectedBailey}), there are two problems on the way (where
$Y_n\ge 0$ stands for the expression under the expectation): Firstly,
$\lim_{n\to\infty} \mean Y_n=0$ does not necessarily imply
$\mean\lim_{n\to\infty} Y_n=0$ since the limit may not exist almost
surely and, secondly, if $\mean\lim_{n\to\infty} Y_n=0$ then
$\lim_{n\to\infty} Y_n=0$ holds almost surely but this equality may
fail on some $1$-random points.

In spite of these difficulties, Ryabko
\cite{Ryabko09} showed in the almost sure setting that there exist
universal forward estimators which are generated by some universal
measures. In view of Theorem \ref{theoBaileyInduced}, this solves the
question of existence of universal measures that induce universal
predictors. But it does not solve yet the problem of effectivization
to algorithmically random points and of systematicity of the
construction.

In this section, we will show that each $1$-universal measure, under a
relatively mild condition (\ref{CondDominationIntro}), satisfied by
the PPM measure, is a $1$-universal forward estimator and hence, in
the light of the previous section, it induces a $1$-universal
predictor. We do not know yet whether this condition is necessary. We
will circumvent Theorem \ref{theoRyabko} by applying Proposition
\ref{theoBreiman} (effective Breiman ergodic theorem) and Theorem
\ref{theoAzuma} (effective Azuma theorem).  The first stage of our
preparations includes two statements which can be called the effective
conditional SMB theorem and the effective conditional universality.
\begin{Proposition}[effective conditional SMB theorem]
    \label{theoExtSMB}
    Let the alphabet be finite and let $P$ be a stationary ergodic
    probability measure. On $1$-$P$-random points we have
 \begin{align}
    \label{ExtSMB}
   \lim_{n\to\infty}
   \frac{1}{n}\sum_{i=0}^{n-1}
   \kwad{-\sum_{x_{i+1}\in\mathbb{X}}P(x_{i+1}|X_1^i)
   \log P(x_{i+1}|X_1^i)}=h_{P}.
 \end{align}
\end{Proposition}
\begin{proof}
  Let us write the conditional entropy
  \begin{align}
    W_i:=\kwad{-\sum_{x_{i+1}\in\mathbb{X}}P(x_{i+1}|X_1^i)
    \log P(x_{i+1}|X_1^i)}.
    \label{CondH}
  \end{align}
  We have $0\le W_i\le \log D$ with $D$ being the cardinality of the
  alphabet. Moreover by Proposition \ref{theoLevy} (effective L\'evy
  law), on $1$-$P$-random points there exists limit
  \begin{align}
    \lim_{n\to\infty} W_n\circ T^{-n-1}=
    \kwad{-\sum_{x_0\in\mathbb{X}}P(x_0|X_{-\infty}^{-1})
    \log P(x_0|X_{-\infty}^{-1})}.
  \end{align}
  Hence by Proposition \ref{theoBreiman} (effective Breiman ergodic
  theorem), on $1$-$(s,P)$-random points 
  \begin{align}
     \lim_{n\to\infty}
    \frac{1}{n}\sum_{i=0}^{n-1} W_i=
    \mean \kwad{-\sum_{x_0\in\mathbb{X}}P(x_0|X_{-\infty}^{-1})
    \log P(x_0|X_{-\infty}^{-1})}=h_{P}
  \end{align}
  since
  $\mean\kwad{-\log
    P(X_0|X_{-\infty}^{-1})}=\lim_{n\to\infty}\kwad{-\log
    P(X_1^n)}/n=h_{P}$.
\end{proof}

\begin{Proposition}[effective conditional universality]
  \label{theoExtUniversal}
  Let the alphabet be finite and let $P$ be a stationary ergodic
  probability measure.
  \begin{align}
    \label{CondDomination}
    -\log R(x_{n+1}|x_1^n)&\le\epsilon_n\sqrt{n/\ln n},
                          \quad \lim_{n\to\infty} \epsilon_n=0
  \end{align}
  then on $1$-$P$-random points we have
  \begin{align}
    \label{ExtUniversal}
    \lim_{n\to\infty}
    \frac{1}{n}\sum_{i=0}^{n-1}
    \kwad{-\sum_{x_{i+1}\in\mathbb{X}}P(x_{i+1}|X_1^i)
    \log R(x_{i+1}|X_1^i)}=h_{P}.
  \end{align}  
\end{Proposition}
\begin{proof}
  Let us write the conditional pointwise entropy
  $Z_i:=-\log R(X_{i+1}|X_1^i)$.
  Now suppose that measure $R$ is $1$-universal and satisfies
  (\ref{CondDomination}). Then by Theorem \ref{theoAzuma} (effective
  Azuma theorem), on $1$-$P$-random points we obtain
\begin{align}
  \lim_{n\to\infty}\frac{1}{n}\sum_{i=0}^{n-1}
  \mean\okra{Z_i\middle|X_1^i}
  =\lim_{n\to\infty}\frac{1}{n}\sum_{i=0}^{n-1}
  Z_i
  =
  \lim_{n\to\infty}\frac{1}{n}\kwad{-\log R(X_1^n)}=h_{P}
  ,
\end{align}
which is the claim of Proposition \ref{theoExtUniversal}.
\end{proof}

In the second stage of our preparations, we recall the famous Pinsker
inequality used by Ryabko \cite{Ryabko09} to prove Theorem
\ref{theoRyabko}.
\begin{Proposition}[Pinsker inequality \cite{CsiszarKorner11}]
  \label{theoPinsker}
  Let $p$ and $q$ be probability distributions over a
  countable alphabet $\mathbb{X}$.  We have
  \begin{align}
    \kwad{\sum_{x\in\mathbb{X}} \abs{p(x)-q(x)}}^2\le
    (2\ln 2) \sum_{x\in\mathbb{X}} p(x)\log\frac{p(x)}{q(x)}.
  \end{align}
\end{Proposition}

Now we can show the main result of this section, namely, that every
universal measure which satisfies a mild condition induces a
universal predictor.
\begin{Theorem}[effective induced prediction II]
  \label{theoUniversal}
  If measure $R$ is $1$-universal and satisfies (\ref{CondDomination})
  then it is a $1$-universal forward estimator.
\end{Theorem}
\begin{proof}
  Let $R$ be a $1$-universal measure, whereas $P$ be the stationary
  ergodic measure.  By Propositions \ref{theoExtSMB} (effective
  conditional SMB theorem) and \ref{theoExtUniversal} (effective
  conditional universality), on $1$-$P$-random points we obtain
  \begin{align}
    \lim_{n\to\infty}
    \frac{1}{n}\sum_{i=0}^{n-1}
    \kwad{\sum_{x_{i+1}}P(x_{i+1}|X_1^i)
    \log\frac{P(x_{i+1}|X_1^i)}{R(x_{i+1}|X_1^i)}}
    =0.
  \end{align}
  Hence by Proposition \ref{theoPinsker} (Pinsker inequality), we
  derive on $1$-$P$-random points
  \begin{align}
    \lim_{n\to\infty}
    \frac{1}{n}\sum_{i=0}^{n-1}
    \kwad{\sum_{x_{i+1}}\abs{P(x_{i+1}|X_1^i)-R(x_{i+1}|X_1^i)}}^2
    =0.
  \end{align}
  Subsequently, the Cauchy-Schwarz inequality
  $\mean Y^2\ge (\mean Y)^2$ yields on $1$-$P$-random points
\begin{align}
  0
  &\ge
  \lim_{n\to\infty}
  \kwad{
  \frac{1}{n}\sum_{i=0}^{n-1}\sum_{x_{i+1}}
  \abs{P(x_{i+1}|X_1^i)-R(x_{i+1}|X_1^i)}}^2  
  \nonumber\\
  &=
    \kwad{
    \lim_{n\to\infty}
    \frac{1}{n}\sum_{i=0}^{n-1}\sum_{x_{i+1}}
    \abs{P(x_{i+1}|X_1^i)-R(x_{i+1}|X_1^i)}}^2
    \ge 0.
\end{align}
Consequently, $R$ is a $1$-universal forward estimator.
\end{proof}

Combining Theorem \ref{theoUniversal} and Theorem
\ref{theoBaileyInduced}, we obtain that predictor $f_R$ is
$1$-universal provided measure $R$ is $1$-universal and satisfies
condition (\ref{CondDomination})---if predictor $f_R$ is computable
itself.  Condition (\ref{CondDomination}) does not seem to have been
discussed in the literature of universal prediction.

\subsection{PPM measure}
\label{secPPM}

In this section, we will discuss the Prediction by Partial Matching
(PPM) measure. The PPM measure comes in several flavors and was
discovered gradually.  Cleary and Witten \cite{ClearyWitten84} coined
the name PPM, which we prefer since it is more distinctive, and
considered the adaptive Markov approximations $\PPM_k$ defined roughly
in equation (\ref{PPMkCond}). Later, Ryabko
\cite{Ryabko88en2,Ryabko08} considered the infinite series $\PPM$
defined in equation (\ref{PPMR}), called it the measure $R$, and
proved that it is a universal measure. Precisely, Ryabko used the
Krichevsky-Trofimov smoothing ($+1/2$) rather than the Laplace
smoothing ($+1$) applied in (\ref{PPMkCond}). This difference does not
affect universality.  As we will show now, series $\PPM$ provides an
example of a $1$-universal measure that satisfies condition
(\ref{CondDomination}) and thus yields a natural $1$-universal
predictor.

Upon the first reading, the definition of the PPM measure may appear
cumbersome but it is roughly a Bayesian mixture of all Markov chains
of all orders. Its universality can be then motivated by the fact that
Markov chains with rational transition probabilities are both
countable and dense in the class of stationary ergodic measures
\cite{Ryabko10b}. Our specific definition of measure $\PPM$ is as
follows.
\begin{Definition}[PPM measure]
  Let the alphabet be $\mathbb{X}=\klam{a_1,..,a_D}$, where $D\ge
  2$. Adapting definitions by
  \cite{ClearyWitten84,Ryabko88en2,Ryabko08,Debowski18}, the PPM
  measure of order $k\ge 0$ is defined as
  \begin{align}
   \label{PPMkCond}
    \PPM_k(x_1^n)
    &:=
      \begin{cases}
      D^{-k-1}\prod_{i=k+2}^n
      \frac{N(x_{i-k}^i|x_1^{i-1})+1}{N(x_{i-k}^{i-1}|x_1^{i-2})+D},
      & k\le n-2,
      \\
      D^{-n}, & k\ge n-1,
      \end{cases}
  \end{align}
  where the frequency of a substring $w_1^k$ in a string $x_1^n$ is
\begin{align}
    N(w_1^k|x_1^n):=\sum_{i=1}^{n-k+1}\boole{x_i^{i+k-1}=w_1^k}.
\end{align}
Subsequently,  we define the total PPM measure
\begin{align}
  \label{PPMR}
  \PPM(x_1^n)
  &:=
  \sum_{k=0}^\infty
  \kwad{\frac{1}{k+1}-\frac{1}{k+2}}\PPM_k(x_1^n)
  .
\end{align}  
\end{Definition}

Infinite series (\ref{PPMR}) is computable since
$\PPM_k(x_1^n)=D^{-n}$ for $k\ge n-1$.  The almost sure universality
of the total PPM measure follows by the Stirling approximation and the
Birkhoff ergodic theorem, see
\cite{Ryabko88en2,Ryabko08,Debowski18}. Since the Birkhoff ergodic
theorem can be effectivized for $1$-random points in form of
Proposition \ref{theoBirkhoff}, we obtain in turn this
effectivization.
\begin{Theorem}[effective PPM universality, cf.\ \cite{Ryabko08}]
  \label{theoPPMUniversal}
  Measure $\PPM$ is $1$-universal.
\end{Theorem}
\begin{proof}
  Computability of the PPM measure follows since series (\ref{PPMR})
  can be truncated with the constant term $\PPM_k(x_1^n)=D^{-n}$ for
  $k\ge n-1$ and thus values $\PPM(x_1^n)$ are rational. To show
  $1$-universality of the PPM measure, we observe the following. Using
  the Stirling approximation, the PPM measure can be related to the
  empirical entropy of the respective string. The empirical
  (conditional) entropy of string $x_1^n$ of order $k\ge 0$ is defined
  as
\begin{align}
  h_k(x_1^n)
  &:=
  \sum_{w_1^{k+1}\in \mathbb{X}^{k+1}}
  \frac{N(w_1^{k+1}|x_1^n)}{n-k}
    \log \frac{N(w_1^k|x_1^{n-1})}{N(w_1^{k+1}|x_1^n)}
  .
\end{align}
In particular, by Theorem A4 in \cite{Debowski18}, we have the bound
\begin{align}
  0\le
  -\log\PPM_k(x_1^n)
  -
  k\log D
  -
  (n-k)h_k(x_1^n)
  \le
  D^{k+1}\log[e^2n]
  .
  \label{PPMbound}
\end{align}
Subsequently, by Proposition \ref{theoBirkhoff} (effective Birkhoff ergodic
theorem), on $1$-$P$-random points we have
\begin{align}
  \lim_{n\to\infty} \frac{N(w_1^{k+1}|X_1^n)}{n-k}=P(w_1^{k+1}).
\end{align}
Then by (\ref{PPMbound}),
\begin{align}
  \lim_{n\to\infty}\frac{1}{n}\kwad{-\log\PPM_k(X_1^n)}=h_{k,P}:=\mean\kwad{-\log
  P(X_{k+1}|X_1^k)}.
\end{align}
Since
\begin{align}
  \label{PPMPPMk}
  -\log\PPM(x_1^n)\le 2\log (k+2)-\log\PPM_k(x_1^n)
\end{align}
then
\begin{align}
  \limsup_{n\to\infty}\frac{1}{n}\kwad{-\log\PPM(X_1^n)}\le
  \inf_{k\ge 0} h_{k,P}=h_{P}
\end{align}
on $1$-$P$-random points, whereas the reverse inequality for the lower
limit follows by Proposition \ref{theoBarron} (effective Barron lemma).
\end{proof}

Finally, we can show that predictor $f_{\PPM}$ induced by the PPM
measure is $1$-universal. First, we notice explicitly these bounds:
\begin{Theorem}[PPM bounds]
  \label{theoPPMDomination}
  We have 
 \begin{align}
  \label{PPMDomination}
   -\log \PPM(x_1^n)&\le 2\log (n+1)+n\log D,
   \\
   \label{PPMCondDomination}
   -\log \PPM(x_{n+1}|x_1^n)&\le 3\log (n+D).
   \end{align}
\end{Theorem}
\begin{proof}
  Observe that $\PPM_k(x_1^n)=D^{-n}$ for $k\ge n-1$. Hence by
  (\ref{PPMPPMk}), we obtain claim (\ref{PPMDomination}).  The
  derivation of claim (\ref{PPMCondDomination}) is slightly longer.
  First, by the definition of $\PPM_k$, we have
  \begin{align}
    -\log \PPM_k(x_{n+1}|x_1^n)\le
    \log \kwad{N(x_{n-k+1}^{n}|x_1^{n-1})+D}
    \le \log(n+D).
  \end{align}
  Now let $G(x_1^n)$ be the minimal $k$ such that $\PPM_k(x_1^n)$ is
  maximal. We have $G(x_1^n)\le n-1$, since $\PPM_k(x_1^n)=D^{-n}$ for
  $k\ge n-1$. Moreover, we have a bound reverse to  (\ref{PPMPPMk}),
  namely
  \begin{align}
    -\log\PPM(x_1^n)\ge -\log\PPM_{G(x_1^n)}(x_1^n).
  \end{align}
  Combining the above with (\ref{PPMPPMk}) yields  
  \begin{align}
    &-\log \PPM(x_{n+1}|x_1^n)
      =
      -\log \PPM(x_1^{n+1})+\log \PPM(x_1^n)
      \nonumber\\
    &\le  
      2\log (G(x_1^n)+2)-\log \PPM_{G(x_1^n)}(x_1^{n+1})
      +\log \PPM_{G(x_1^n)}(x_1^n)
      \nonumber\\
    &=
      2\log (G(x_1^n)+2)
      +\log \PPM_{G(x_1^n)}(x_{n+1}|x_1^n)
      \le  
      3\log (n+D).
  \end{align}
  
\end{proof}

Now comes the main theorem.
\begin{Theorem}
  \label{theoPPMUniversalII}
  Predictor $f_{\PPM}$ is $1$-universal.
\end{Theorem}
\begin{proof}
  Computability of predictor $f_{\PPM}$ follows since values
  $\PPM(x_1^n)$ are rational so the least symbol of those having the
  maximal conditional probability can be computed in a finite
  time. Consequently, the claim follows by Theorems
  \ref{theoUniversal}, \ref{theoPPMUniversal}, and
  \ref{theoPPMDomination}.
\end{proof}

We think that $1$-universality of predictor $f_{\PPM}$ is quite
expected and intuitive.  But as we can see, the PPM measure satisfies
condition (\ref{CondDomination}) with a large reserve.  It is an open
question whether there are $1$-universal measures such that
conditional probabilities $R(x_{n+1}|x_1^n)$ converge to zero much
faster than for the PPM measure but they still induce $1$-universal
predictors. It would be interesting to find such measures. Maybe they
have some other desirable properties, also from a practical point of
view.

\section*{Funding}

This work was supported by the National Science Centre Poland
grant 2018/31/B/HS1/04018.

\section*{Acknowledgments}

The authors are grateful to the anonymous reviewers of unaccepted
earlier conference versions of the paper, who provided a very
stimulating and encouraging feedback.  Additional improvements to the
paper were inspired by participants of the Kolmogorov seminar in
Moscow at which this work was presented by the first author. Finally,
the authors express their gratitude to Dariusz Kalociński for his
comments and proofreading. Both authors declare an equal contribution
to the paper.

\bibliographystyle{plainurl}

\bibliography{coding_prediction_arxiv}

\end{document}